\documentclass[11pt,english,a4paper]{smfart}
\usepackage[english]{babel}
\usepackage{amssymb,xspace}
\usepackage{amstext}
\usepackage{smfthm}
\theoremstyle{plain}
\usepackage{amsbsy,amssymb,amsfonts,latexsym}

\usepackage{enumerate}
\usepackage{graphics}

\date{\today}

\title[An example of a minimal action of $\F^{+}_{2}$ on the Hilbert space]{An example of a minimal action of the free semi-group $\F^{+}_{2}$ on the Hilbert space}

\author{Sophie Grivaux}
\address{CNRS,
Laboratoire Paul Painlev\' e, UMR 8524, Universit\'e  Lille 1, Cit\' e Scientifique, 59655 Villeneuve d'Ascq
Cedex, France}
\email{grivaux@math.univ-lille1.fr}

\author{Maria Roginskaya}
\address{Department of Mathematical Sciences,
Chalmers University of Technology, SE-41296 G\"oteborg, Sweden, \emph{and}
Department of Mathematical Sciences,
G\"oteborg University, SE-41296 G\"oteborg, Sweden}
\email{maria@chalmers.se}



\def\D{\ensuremath{\mathbb D}}

\def\T{\ensuremath{\mathbb T}}

\def\C{\ensuremath{\mathbb C}}

\def\N{\ensuremath{\mathbb N}}

\def\F{\ensuremath{\mathbb F}}

\DeclareMathOperator{\vect}{span}

\newcommand{\sep}{separable}

\newcommand{\hy}{hypercyclic}

\newcommand{\ops}{operators}
\newcommand{\op}{operator}

\newcommand{\nz}{non-zero}

\newcommand{\pss}[2]{\ensuremath{{\langle #1,#2\rangle}}}

\newtheorem{theorem}{Theorem}[section]

\newtheorem{lemma}[theorem]{Lemma}

\newtheorem{proposition}[theorem]{Proposition}

\newtheorem{corollary}[theorem]{Corollary}

{\theoremstyle{definition}}

{\theoremstyle{definition}}

{\theoremstyle{definition}}

{\theoremstyle{definition}}

{\theoremstyle{definition}}

{\theoremstyle{definition}\newtheorem*{FFC Criterion}{Frequent
Faber-hypercyclicity Criterion}}
\newtheorem*{Hypercyclicity Criterion}{Hypercyclicity Criterion}
{\theoremstyle{definition}\newtheorem*{GS Criterion}{Godefroy-Shapiro
Criterion}}
\def\piednote#1{\let\oldfn=\thefootnote\def\thefootnote{}\footnote{\noindent#1}%
\addtocounter{footnote}{-1}\def\thefootnote{\oldfn}}

\def\wh{\widehat}
\usepackage[all]{xy}
\entrymodifiers={+!!<0pt,\fontdimen22\textfont2>}

\begin{document}

\begin{abstract}
The Invariant Subset Problem on the Hilbert space is to know whether there exists a bounded linear operator $T$ on a separable infinite-dimensional Hilbert space $H$ such that the orbit $\{T^{n}x;\ n\ge 0\}$ of every non-zero vector $x\in H$ under the action of $T$ is dense in $H$.
We show that there exists a bounded linear operator $T$ on a complex separable infinite-dimensional Hilbert space $H$ and a unitary operator $V$ on $H$, such that the following property holds true: for every non-zero vector $x\in H$, either $x$ or $Vx$ has a dense orbit under the action of $T$. As a consequence, we obtain in particular that there exists a minimal action of the free semi-group with two generators $\F^{+}_{2}$ on a complex separable infinite-dimensional Hilbert space $H$.
\end{abstract}
\maketitle
\section{Introduction}\label{Intro}
Let $H$ be a complex separable infinite-dimensional Hilbert space, and $T\in\mathcal{B}(H)$ a bounded linear operator on $H$. Our present work is motivated by the well-known Invariant Subspace and Subset Problems, which can be easily stated as follows: if $T\in\mathcal{B}(H)$, does there always exist a closed subspace $M$ of $H$ (resp. a closed subset $F$ of $H$), non-trivial in the sense that it is distinct from $\{0\}$ and $H$, and which is invariant under the action of $T$? It is not difficult to see that an \op\ $T$ on $H$ has no non-trivial invariant closed subspace (resp. subset) if and only if for every non-zero vector $x\in H$ the linear span of the orbit $\textrm{Orb}(x,T)=\{T^{n}x;\ n\ge 0\}$ of $x$ under the action of $T$ (resp. the orbit $\textrm{Orb}(x,T)$ itself) is dense in $H$. The Invariant Subset Problem can thus be reformulated as follows: does there exist a minimal action of the semi-group $\N$ on the Hilbert space $H$?
\par\smallskip 
 Of course the Invariant Subspace and Subset Problems make sense on general separable Banach spaces as well, and in this setting both problems admit a negative answer. Enflo \cite{En} and Read \cite{R1} constructed examples of infinite-dimensional \sep\ Banach spaces $X$ and of bounded operators $T$ on $X$ without non-trivial invariant closed subspaces. Read then refined his constructions in several papers such as \cite{R2}, where he gave an example of an operator on $\ell_{1}$ without non-trivial invariant closed subspaces, \cite{R3} which exhibited an example of an operator on $\ell_{1}$ without non-trivial invariant closed subsets, and \cite{R4} which gives examples of operators without non-trivial invariant subspaces on $c_{0}$ or on the $\ell_{2}$-direct sum of countably many copies of the James space $J$. All these counterexamples are constructed on non-reflexive spaces. Both the Invariant Subspace and Subset Problems remain open on reflexive spaces, and on the Hilbert space in particular. We refer the reader to the survey \cite{CE}, the book \cite{RR} or to the recent book \cite{CP} for more information. These references are mainly concerned with results in the positive direction, i.e. conditions under which an operator does admit a non-trivial invariant subspace (or subset).
\par\medskip 
At present, operators on the Hilbert space which seem ``closest'' to having no non-trivial invariant closed subset are the ones constructed in the paper \cite{GR}. 
 We call such operators \emph{Read's type operators} because they are constructed by adapting part of the techniques employed by Read in his various constructions to the Hilbert space setting.
 In order to state the properties of these operators which will be needed in this paper, let us recall that
a vector $x$ of $H$ is called \emph{hypercyclic} for an \op\ $T$ if the orbit $\textrm{Orb}(x,T)$ is dense in $H$. We denote by $HC(T)$ the set of such vectors. Then $T$ has no non-trivial invariant closed subset if and only if $HC(T)^{c}=\{0\}$. The operators constructed in \cite{GR} have the property that $HC(T)^{c}$ is ``small'', in the sense that it is contained in a countable union of closed hyperplanes of $H$.
\par\medskip 
Our main result in this paper is the following:
\begin{theorem}\label{Th1}
 There exists a bounded linear operator $T$ acting on a complex separable infinite-dimensional Hilbert space $H$, and a unitary operator $V$ on $H$, such that the following property holds true:
for every non-zero vector $x$ of $ H$, either $x$ or $Vx$ has a dense orbit under the action of $T$. 
\end{theorem}
In other words, $HC(T)\cup HC(V^{-1}TV)=H\setminus \{0\}$. In particular, $\bigcup_{n\ge 0}T^{n}\{x,Vx\}$ is dense in $H$ for every non-zero vector $x\in H$.
\par\smallskip 
This theorem has the following interesting consequence: denote by $\F^{+}_{2}$ the free semi-group with two generators. It can be seen as the set of all finite sequences 
$(\omega _{0},\dots,\omega _{n})$, $n\ge 0$, of zeroes and ones, where the group operation is given by concatenation: 
\[
\textrm{if}\ \ \omega =(\omega _{0},\dots,\omega _{n})\ \ \textrm{and}\ \ \theta =(\theta _{0},\dots,\theta _{p})\ \ \textrm{then}\ \ \omega \,.\,\theta =(\omega _{0},\dots,\omega _{n},\theta _{0},\dots,\theta _{p}).
\]
If $T_{0}$ and $T_{1}$ are two bounded linear operators on $H$, one can define an action $\rho $ of $\F^{+}_{2}$ on $H$ in the following way:
\[
\xymatrix@C=50pt@R=2pt{\hphantom{aaaa}\rho :\F^{+}_{2}\times H\ar[r]&H\hphantom{aaaaaaaa}\\
\bigl( (\omega _{0},\dots,\omega _{n}),x\bigr)\ar@{|->}[r]&T_{\omega _{0}}\dots T_{\omega _{n}}x.}
\]
With the notation above, $\rho (\omega \, .\, \theta ,x)=T_{\omega _{0}}\dots T_{\omega _{n}}T_{\theta _{0}}\dots T_{\theta _{p}}x=\rho (\omega ,\rho (\theta ,x))$. Such an action is called \emph{minimal} if the orbit of every non-zero vector $x$ of $ H$, that is the set 
$\{\rho (\omega ,x),\ \omega \in \F^{+}_{2}\}$, is dense in $H$.
\par\medskip 
If $T$ and $V$ are the operators given by Theorem \ref{Th1} above, define two operators $T_{0}$ and $T_{1}$ on $H$ by setting $T_{0}=T$ and $T_{1}=V^{-1}TV$. Let $\rho $ be the associated action of $\F^{+}_{2}$ on $H$. For every $x\in H\setminus \{0\}$, the orbit $\{\rho (\omega ,x),\ \omega \in \F^{+}_{2}\}$ contains the sets 
$\{T^{n}_{0}x,\ n\ge 0\}$ and $\{T^{n}_{1}x,\ n\ge 0\}$. Since one of these two sets is dense in $H$, we obtain in particular that $\rho $ is minimal.
\begin{corollary}\label{Cor2}
 There exists a minimal action of the free semi-group with two gene\-rators $\F^{+}_{2}$ on a complex separable infinite-dimensional Hilbert space $H$.
\end{corollary}
The operator $T$ of Theorem \ref{Th1} is a Read's type operator. Such operators are non-invertible, and this seems to be inherent to the construction. So it is not clear whether it is possible to construct a minimal action of the free group $\F_{2}$ on a Hilbert space. But if we replace the operators $T_{0}=T$ and $T_{1}=V^{-1}TV$ above by operators of the form $S_{0}=\alpha I+T$ and $S_{1}=\alpha I+V^{-1}TV$, where $\alpha \in\C$ is so large that $\alpha I+T$ is invertible, then for every non-zero vector $x\in H$ the linear span of one of the two sets $\{S^{n}_{0}x,\ n\ge 0\}$ and $\{S^{n}_{1}x,\ n\ge 0\}$ is dense in $H$.
So we obtain in the same way an action $\rho $ of $\F_{2}$ on $H$ such  that for every  non-zero vector $x$ of $ H$, the linear span of the orbit $\{\rho (\omega ,x),\ \omega \in \F_{2}\}$ of $x$ under the action of $\F_{2}$ by $\rho $ is dense in $H$.
\par\smallskip 
Minimal affine isometric actions of groups on an infinite-dimensional Hilbert space are studied in the paper \cite{CTV}, and it is proved here in particular that there exists a minimal isometric action of the free groupe with three generators $\mathbb{F}_{3}$ on an infinite-dimensional Hilbert space.
\par\smallskip
The proof of Theorem \ref{Th1} uses some fine properties of the operators constructed in \cite{GR}. After recalling these briefly in Section \ref{Sec2} we show, using a slightly modified version of the Lomonosov inequality of \cite{L}, that these Read's type operators on the Hilbert space do have non-trivial invariant closed subspaces, and we give a rather precise description of their set of non-hypercyclic vectors. We prove Theorem \ref{Th1} in Section \ref{Sec3}, using this description as well as a construction of specific increasing sequences of subspaces in the vector-valued $H^{2}$-space $H^{2}(\D,\ell^{2})$. We also observe there that Theorem \ref{Th1} yields an example of a pair of two unitarily equivalent operators on 
$H$ which generate $\mathcal{B}(H)$ in the strong (or weak) topology.

\section{Main properties of Read's type operators on the Hilbert space}\label{Sec2}

The main properties of the \ops\ constructed in \cite{GR} which will be of interest to us here are summarized in the following theorem:

\begin{theorem}\label{Th3}
Let $H$ be a separable infinite-dimensional Hilbert space. There exist bounded \ops\ $T$ on $H$ having the following three properties:
\begin{enumerate}
 \item [\emph{(P1)}] for every vector $x\in H$, the closure of the orbit $\textrm{Orb}(x,T)$ of $x$ under the action of $T$ is a subspace. In other words, the closures of $\{T^{n}x,\ n\ge 0\}$ and $\vect \{T^{n}x,\ n\ge 0\}$ coincide;
\item[\emph{(P2)}] the family $(\overline{\textrm{Orb}}(x,T))_{x\in H}$ of all the closures of the orbits of $T$ is totally ordered by inclusion: for every pair $(x,y)$ of vectors of $H$, either  $\overline{\textrm{Orb}}(x,T)\subseteq\overline{\textrm{Orb}}(y,T)$ or $\overline{\textrm{Orb}}(y,T)\subseteq\overline{\textrm{Orb}}(x,T)$;
\item[\emph{(P')}] the \op\ $T_{|M}$ induced by $T$ on any of its invariant subspaces $M$ is \hy, i.e.\ there exists a vector $x\in M$ such that $\overline{\textrm{Orb}}(x,T)=M$. 
\end{enumerate}
\end{theorem}
Property (P') is an easy consequence of properties (P1) and (P2) (see \cite[Section 5.1]{GR}), but we state it explicitly here as it will be needed in the sequel. Another observation which will be important is the following: if $(g_{j})_{j\ge 0}$ denotes an orthonormal basis of $H$, the \ops\ $T$ of \cite{GR} which are constructed starting from this basis have the form $T=S+K$, where $S$ is a forward weighted shift with respect to the basis $(g_{j})_{j\ge 0}$ and $K$ is a nuclear \op: $Sg_{j}=w_{j}\, g_{j+1}$ where $0\le w_{j}\le 2$ for all $j$ and $w_{j}$ is either $0$ or very close to $1$ as $j$ tends to infinity, and $\sum_{j\ge 0}||Kg_{j}||$ can be made as small as we wish. This kind of property is common to all Read's type operators: in the definition of the vectors $Tg_{j}$ for $j\ge 0$, the set of all nonnegative integers $j$ is partitioned into two types of intervals, lay-off intervals and working intervals. Working intervals are separated by very long lay-off intervals. When $j$ is not the right endpoint of either a working interval or a lay-off interval, $Tg_{j}$ is defined as $Tg_{j}=w_{j}g_{j+1}$, where the weight $w_{j}$ is chosen extremely close to $1$ when $j$ gets very large. This gives the shift part $S$ of the \op\ $T$. When $T$ is a right endpoint of a lay-off or a working interval, $||Tg_{j}||$ is extremely small and can be chosen to decrease very quickly as $j$ grows. This gives the nuclear part $K$ of the \op, and $T=S+K$. A consequence of this observation is that such \ops\ $T$ do have non-trivial invariant subspaces, although, as will be recalled shortly afterwards, their sets of non-\hy\ vectors are very small.

\begin{proposition}\label{Prop4}
 If $T$ is one of the operators of \cite{GR} acting on a complex separable infinite-dimensional Hilbert space $H$, $HC(T)^{c}$ is a dense linear subspace of $H$, and $T$ has non-trivial invariant closed subspaces.
\end{proposition}

The proof of Proposition \ref{Prop4} relies on a refinement of the Lomonosov inequality proven in \cite{L}. This inequality motivates the well-known Lomonosov conjecture that any adjoint \op\ acting on a complex dual separable Banach space $X^{*}$ has a non-trivial invariant closed subspace. It is stated as follows:
\par\smallskip 
Let $X$ be a complex separable Banach space, and let $\mathcal{A}$ be a weakly closed sub-algebra of $\mathcal{B}(X)$ with $\mathcal{A}\neq\mathcal{B}(X)$. Then there exist two non-zero vectors $x^{**}\in X^{**} $ and $x^{*}\in X^{*}$ such that for every $A\in \mathcal{A}$,
$$
|\pss{x^{**}}{A^{*}x^{*}}|\le||A||_{e},
$$
where $||A||_{e}$ denotes the essential norm of $A$, i.e.\ the distance of $A$ to the space of compact operators on $X$.
\par\smallskip 
Here is the slightly stronger statement which will be needed in the present paper. It is completely contained in the proof of \cite[Th. $1$]{L}, and so we will only say a few words about it.

\begin{theorem}\label{th0}
 Let $X$ be a complex separable Banach space, and let $\mathcal{A}$ be a uniformly closed subalgebra of $\mathcal{B}(X)$.
 Let $Q$ be a closed ball of $X^{*}$ of positive radius, which does not contain the point $0$.
 Then we have the following alternative: 
 either
 
(A1)  there exist a positive constant $C_{Q}$ depending only on $Q$, a vector $x^{*}\in Q$ and a \nz\ vector 
 $x^{**}\in X^{**}$ such that 
\begin{eqnarray*}\label{star}
|\pss{x^{**}}{A^{*}x^{*}}|\le C_{Q}\,||A||_{e}\quad \textrm{ for every } A\in \mathcal{A}
\end{eqnarray*}

or

(A2)  the set $\{A^{*}x^{*} \textrm{ ; } A\in \mathcal{A}\}$ is dense in $X^{*}$ for every vector $x^{*}\in Q$, and in this case there exists an \op\ $A_{0}^{*}$ on $X$ belonging to $\mathcal{A}$, distinct from the identity \op, and such that $A_{0}^{*}x_{0}^{*}=x_{0}^{*}$ for some vector $x_{0}^{*}$ in $Q$.
\end{theorem}

\begin{proof}
The proof of Lemma $8$ of \cite{L} gives the following:
if $y^{*}\in X^{*}$ is such that $||y^{*}||=3$, and if $Q$ denotes the ball of radius $2$ centered at $y^{*}$, then either there exist a vector $x^{*}\in Q$ and a \nz\ element $x^{**}$ of $X^{**}$ such that
\begin{eqnarray*}
|\pss{x^{**}}{A^{*}x^{*}}|\le 10\,||A||_{e}\quad \textrm{ for every } A\in \mathcal{A},
\end{eqnarray*}
  or the set $\{A^{*}x^{*} \textrm{ ; } A\in \mathcal{A}\}$ is dense in $X^{*}$ for every vector $x^{*}\in Q$, and in this case there exists an element $A_{0}$ of $\mathcal{A}$ such that  $A_{0}^{*}x_{0}^{*}=x_{0}^{*}$ for some vector $x_{0}^{*}$ in $Q$, and $1$ is an eigenvalue of $A_{0}^{*}$ of finite multiplicity which is an isolated point in the spectrum of $A_{0}^{*}$. 
Hence (A2) is satified when $Q$ is the closed ball centered at $y^{*}$ of radius $2$

Exactly the same proof shows that if $y^{*}\in X^{*}$ is \nz\ and $Q$ is the ball of radius $r$ centered at $y^{*}$, with $0<r<||y^{*}||$, then either there exist a vector $x^{*}\in Q$ and a \nz\ element $x^{**}$ of $X^{**}$ such that
\begin{eqnarray*}
|\pss{x^{**}}{A^{*}x^{*}}|\le 2(||y^{*}||+r)\,||A||_{e}\quad \textrm{ for every } A\in \mathcal{A},
\end{eqnarray*}
or (A2) holds true. This is the statement of Theorem \ref{th0}.
\end{proof}

Let us now mention an important consequence of Theorem \ref{th0}, which is a slight generalization of a result of \cite{JPK}:

\begin{corollary}\label{cor}
Suppose that $T$ is a bounded \op\ on a complex separable Hilbert space which is a compact perturbation of a power-bounded \op\ on $H$. Then $T$  has a dense set of non-\hy\ vectors.
\end{corollary}

\begin{proof}
By our assumption the \op\ $T$ can be decomposed as $T=B+K$, where
$\sup_{n\ge 0}||B^{n}||\le M<+\infty $ and $K$ a compact operator on $H$. Then $T^{n}$ is a compact perturbation of $B^{n}$ and $||T^{*n}||_{e}=||T^{n}||_{e}=||B^{n}||_{e}\le M$. Let $\mathcal{A}$ be the uniformly closed sub-algebra of $\mathcal{B}(H)$ generated by $T^{*}$.
Then for every closed ball $Q$ of $H$ with non-empty interior not containing $0$, either (A1) or (A2) of Theorem \ref{th0} above is satisfied.  But it is not difficult to see that (A2) can never hold here. Suppose indeed that (A2) is true: then there exist an \op\ $A_{0}\in \mathcal{A}$, not equal to the identity, and a vector $x_{0}\in Q$ such that $A_{0}^{*}x_{0}=x_{0}$. Since the algebra $\mathcal{A}$ is commutative, the equality $A^{*}x_{0}=A^{*}A_{0}^{*}x_{0}=A_{0}^{*}A^{*}x_{0}$  holds true for every $A\in \mathcal{A}$.
But $x_{0}$ belongs to $Q$, so the set $\{A^{*}x_{0} \textrm{ ; } A\in \mathcal{A}\}$ is dense in $H$ by (A2). Hence $A_{0}=I$, which is a contradiction. So (A2) cannot be true, and (A1) is satisfied: for every non-empty open ball $Q$ of $H$ not containing $0$,
there exist  a positive constant $C_{Q}$ and two \nz\ vectors $x,y\in H$ with $y\in Q$ such that
$\sup_{n\ge 0}|\pss{x}{T^{n}y}|\le C_{Q}\, M$. Such vectors $y$ cannot be \hy\ for $T$, so $T$ has a dense set of non-\hy\ vectors.
\end{proof}

The proof of Proposition \ref{Prop4} is now straightforward.

\begin{proof}[Proof of Proposition \ref{Prop4}] 
If $T$ is one of the \ops\ of \cite{GR}, let $(x_{\alpha })_{\alpha \in A}$ be the family of all vectors of $H$ which are non-\hy\ for $T$. If for each $\alpha \in A$ we denote by $M_{\alpha }$ the subspace $M_{\alpha }=\overline{\textrm{Orb}}(x_{\alpha },T) $, which is non-trivial, then 
\[
HC(T)^{c}=\bigcup_{\alpha \in A}M_{\alpha }.
\]
Since the subspaces $M_{\alpha }$ are totally ordered by inclusion, the set $HC(T)^{c}$ is a linear space. It remains to prove that it is dense in $H$. We have seen that $T$ can be decomposed as
$T=S+K$, where the weights $w_{j}$ of the shift are either $0$ or tend to $1$ very quickly as $j$ tends to infinity and $K$ is compact. Let $S_{0}$ be the weighted shift on $H$ defined by 
$S_{0}g_{j}=0$ if $\ w_{j}=0$ and $S_{0}g_{j}=g_{j+1}$ if $w_{j}>0$. Then $S-S_{0}=L$ is a compact \op. So $T=S_{0}+L+K$ is a compact perturbation of $S_{0}$, and $S_{0}$ is obviously  power-bounded since $||S_{0}||\le 1$. Hence $T$ has a dense set of non-\hy\ vectors by Corollary \ref{cor}. Since all invariant closed subsets of $H$ are automatically subspaces by property (P1) of Theorem \ref{Th3}, $T$ has  non-trivial invariant closed subspaces.
 \end{proof}
 
Thanks to Proposition \ref{Prop4}, it is possible to obtain a rather complete description of the structure of the set of non-\hy\ vectors for these operators.
\begin{proposition}\label{Prop6}
 If $T$ is one of the operators of \cite{GR} on a complex Hilbert space $H$, there exists an increasing sequence $(M_{n})_{n\ge 0}$ of infinite-dimensional closed subspaces of $H$ such that $M_{n}$ has infinite codimension as a subspace of $M_{n+1}$ for each $n\ge 0$, $\bigcup_{n\ge 0}M_{n}$ is dense in $H$, and $HC(T)^{c}=\bigcup_{n\ge 0}M_{n}$.
\end{proposition}
\begin{proof}
 With the notation used in the proof of Proposition \ref{Prop4} above, $HC(T)^{c}=\bigcup_{\alpha \in A}M_{\alpha }$ with $M_{\alpha }=\overline{\textrm{Orb}}(x_{\alpha },T)$. Let $(x_{\alpha _{k}})_{k\ge 0}$ be a  sequence of vectors of $HC(T)^{c}$ which is dense in $HC(T)^{c}$. Since $HC(T)^{c}$ is dense in $H$, this sequence is dense in $H$ as well. It is proved in \cite[Section 5.2]{GR} that 
$$
HC(T)^{c}=\bigcup_{k\ge 0}M_{\alpha_{k} }.
$$
Since the argument is simple enough, we recall it briefly here: let 
$x_{\alpha }\in HC(T)^{c}$. We wish to show that there exists a $k\ge 0$ such that $M_{\alpha }\subseteq M_{\alpha _{k}}$. Suppose that this is not the case. Then $M_{\alpha _{k}}\subseteq M_{\alpha }$ for each $k\ge 0$ by property (P2), so that $x_{\alpha _{k}}\in M_{\alpha }$ for each $k\ge 0$. Since $\{x_{\alpha _{k}} \textrm{ ; } k\ge 0\}$ is dense in $HC(T)^{c}$ and $M_{\alpha }$ is closed, it follows that the closure of $HC(T)^{c}$ is contained in $ M_{\alpha }$, which is an obvious contradiction with Proposition \ref{Prop4} and the fact that $M_{\alpha }\neq H$. Hence $HC(T)^{c}=\bigcup_{k\ge 0}M_{\alpha _{k}}$.
\par\medskip 
Set $M_{0}=M_{\alpha _{0}}$. Let then $k_{1}$ be the smallest integer such that $M_{0}\subsetneq M_{\alpha _{k_{1}}}$. Such an integer does exist because the subspaces $M_{\alpha _{k}}$ are totally ordered by inclusion and it is impossible that $M_{\alpha _{k}}\subseteq M_{0}$ for each $k\ge 1$. Set then $M_{1}=M_{\alpha _{k_{1}}}$. We have $M_{\alpha _{1}}\subseteq M_{0},\, \dots,\,M_{\alpha _{k_{1}-1}}\subseteq M_{0}$, so that 
$\bigcup_{k= 0}^{k_{1}}M_{\alpha _{k}}=M_{1}$. Continuing in this fashion, we construct by induction a strictly increasing sequence $(k_{n})_{n\ge 1}$ of integers such that, setting $M_{n}=M_{\alpha_{k_{n}}}$, $k_{n}$ is the smallest integer such that $M_{n-1}=M_{\alpha_{k_{n-1}}}\subsetneq M_{\alpha_{k_{n}}}=M_{n}$. With this definition we have as above $\bigcup_{k= 0}^{k_{n}}M_{\alpha _{k}}=M_{n}$, $M_{n}
\subsetneq M_{n+1}$, and  
$\bigcup_{k\ge 0}M_{\alpha _{k}}=\bigcup_{n\ge 0} M_{n}=HC(T)^{c}$.
\par\medskip 
Each subspace $M_{n}$ of $H$ is infinite-dimensional by property (P') and the fact that no \op\ on a finite-dimensional space admits a \hy\ vector. It remains to prove that for each $n\ge 0$, $M_{n}$ is a subspace of $M_{n+1}$ of infinite codimension in $M_{n+1}$: the \op\ $T_{|M_{n+1}}$ induced by $T$ on $M_{n+1}$ admits a \hy\ vector by property (P'), so the quotient \op\ $\overline{T}_{|M_{n+1}}$ on the non-zero quotient space $M_{n+1}/M_{n}$ admits a \hy\ vector as well. So $M_{n+1}/M_{n}$ is necessarily infinite-dimensional, and this proves our claim. 
\end{proof}
As a consequence of Proposition \ref{Prop6} we obtain:
\begin{proposition}\label{Prop7}
 Let $(E_{n})_{n\ge 0}$ be a sequence of closed subspaces of the Hilbert space $H$ having the following properties:
\begin{enumerate}
 \item [(1)] $E_{n}\subseteq E_{n+1}$ for each $n\ge 0$,
\item[(2)] each subspace $E_{n}$ is infinite-dimensional and has infinite codimension as a subspace of $E_{n+1}$,
\item[(3)] the union $\bigcup_{n\ge 0}E_{n}$ is dense in $H$.
\end{enumerate}
Let $T$ be one of the \ops\ of \cite{GR}, having the property given by Proposition \emph{\ref{Prop6}}. There exists a unitary operator $U$ on $H$ such that 
\[
HC(UTU^{-1})^{c}=\bigcup_{n\ge 0}E_{n}.
\]
\end{proposition}
\begin{proof}
 Let $T$ be one of these \ops, and let $(M_{n})_{n\ge 0}$ be the increasing sequence of subspaces such that $HC(T)^{c}=\bigcup_{n\ge 0}M_{n}$ given by Proposition \ref{Prop6}. Let $(g_{k,0})_{k\ge 0}$ be an orthonormal basis of $M_{0}$, and for each $n\ge 1$ let $(g_{k,n})_{k\ge 0}$ be an orthonormal basis of the orthogonal complement
 $M_{n}\ominus M_{n-1}$ of $M_{n-1}$ in $M_{n}$. If $(E_{n})_{n\ge 0}$ is a sequence of subspaces having the properties stated in Proposition \ref{Prop7}, let $(e_{k,0})_{k\ge 0}$ be an orthonormal basis of $E_{0}$ and $(e_{k,n})_{k\ge 0} $ an orthonormal basis of $E_{n}\ominus E_{n-1}$ for each $n\ge 1$. As each one of the two sets $\{g_{k,n} \textrm{ ; } n\ge 0,\ k\ge 0\}$ and $\{e_{k,n} \textrm{ ; } n\ge 0,\ k\ge 0\}$ span a dense subspace of $H$, each of them forms an orthonormal basis of $H$, and there exists a unitary operator  $U$ on $H$ such that $Ug_{k,\,n}=e_{k,\,n}$ for each $n\ge 0$ and $k\ge 0$. It follows that $U(M_{n})=E_{n}$ for each $n\ge 0$, and thus the \op\ $UTU^{-1}$ satisfies $HC(UTU^{-1})^{c}=\bigcup_{n\ge 0}E_{n}$.
\end{proof}

We are now ready for the proof of Theorem \ref{Th1}.

\section{Proof of Theorem \ref{Th1}}\label{Sec3}
The proof of Theorem \ref{Th1} relies on the following simple idea: let $(E_{n})_{n\ge 0}$ and $(F_{n})_{n\ge 0}$ be two sequences of closed subspaces of $H$ satisfying properties (1), (2), and (3) of Proposition \ref{Prop7}. Let $T\in\mathcal{B}(H)$ be one of the \ops\ of \cite{GR}, and let $U_{E}$ and $U_{F}$ be the two associated unitaries given by Proposition \ref{Prop7}: 
$$HC(U_{E}TU_{E}^{-1})^{c}=\bigcup_{n\ge 0}E_{n} \quad \textrm{and}\quad HC(U_{F}TU_{F}^{-1})^{c}=\bigcup_{n\ge 0}F_{n}.$$ If we manage to construct the sequences of subspaces $(E_{n})_{n\ge 0}$ and $(F_{n})_{n\ge 0}$ in such a way that 
$$
\Bigl( \bigcup_{n\ge 0}E_{n}\Bigr)\cap\Bigl( \bigcup_{n\ge 0}F_{n}\Bigr)=\{0\},
$$
then we will get that $HC(U_{E}TU_{E}^{-1})^{c}\cap HC(U_{F}TU_{F}^{-1})^{c}=\{0\}$. This means that 
$$HC(T)^{c}\cap HC(V^{-1}TV)^{c}=\{0\},$$ where $V=U_{F}^{-1}U_{E}$, i.e.\ that $HC(T)\cup HC(V^{-1}TV)=H\setminus \{0\}$, and this is exactly the statement of Theorem \ref{Th1}. So we see that everything boils down to showing the following lemma:
\begin{lemma}\label{Lem8}
 Let $H$ be a complex separable infinite-dimensional Hilbert space. There exist two sequences $(E_{n})_{n\ge 0}$ and $(F_{n})_{n\ge 0}$ of closed subspaces of $H$ which both satisfy assertions \emph{(1)}, \emph{(2)}, and \emph{(3)} of Proposition  \emph{\ref{Prop7}} and are such that
\[
\Bigl( \bigcup_{n\ge 0}E_{n}\Bigr)\cap\Bigl( \bigcup_{n\ge 0}F_{n}\Bigr)=\{0\}.
\]
\end{lemma}
\begin{proof}
 Let us first work in the Hardy space $H^{2}(\D)$, and consider for $n\ge 0$ the spaces $G_{n}=\vect[1,z,\dots,z^{n}]$ of polynomials of degree at most $n$. Of course, $G_{n}\subseteq G_{n+1}$ and $\dim(G_{n+1}\ominus G_{n})=1$ for each $n\ge 0$, and $\bigcup _{n\ge 0}G_{n}$ is dense in $H^{2}(\D)$. Let now $(z_{j})_{j\ge 0}$ be the sequence of points of $\D$ defined by $z_{j}=1-1/(j+1)^{2}$ for $j\ge 0$. For each $n\ge 0$, consider the Blaschke product $B_{n}$ given by 
$$
B_{n}(z)=\prod_{j\ge n}\dfrac{z_{j}-z}{1-z_{j}z}\quad (z\in\D)
$$
and the subspaces $K_{n}=B_{n}H^{2}(\D)$ of $H^{2}(\D)$. This is an increasing sequence of infinite-dimensional closed subspaces of $H^{2}(\D)$, and $\dim(K_{n+1}\ominus K_{n})=1$ for each $n\ge 0$. Let us now show that the increasing union $\bigcup_{n\ge 0}K_{n}$ is dense in $H^{2}(\D)$. Let $f\in H^{2}(\D)$ be a function which is orthogonal to all subspaces $K_{n}$. Denote by $L^{2}_{-}(\T)$ the set of functions $g$ in $L^{2}(\T)$ such that $\wh{f}(j)=0$ for each $j\ge 0$. Saying that f belongs to $H^{2}(\D)\ominus B_{n}H^{2}(\D)$ means that $f\in H^{2}(\D)\cap B_{n}L^{2}_{-}(\T)$. So  there exists for each $n\ge 0$ a function $g_{n}\in L^{2}_{-}(\T)$ such that $f=B_{n}g_{n}$. This means that $f(e^{i\theta })=B_{n}(e^{i\theta })g_{n}(e^{i\theta })$ for almost every $e^{i\theta }\in\T$, and hence there exists a subset $\Omega $ of $\T$ of measure $1$ such that for every $e^{i\theta }\in\Omega $ and every $n\ge 0$, $f(e^{i\theta })=B_{n}(e^{i\theta })g_{n}(e^{i\theta })$. Now the estimates
\begin{align*}
 |1-B_{n}(e^{i\theta })|&\le \sum_{j\ge n}\,\Bigl| \dfrac{z_{j}-e^{i\theta }}{1-z_{j}e^{i\theta }}-1\Bigr|\le\sum_{j\ge n}(1-z_{j})\, \dfrac{2}{|1-z_{j}e^{i\theta }|}\\
&\le\dfrac{2}{|1-e^{i\theta }|-\frac{1}{(n+1)^{2}}}\, \sum_{j\ge n}\dfrac{1}{(j+1)^{2}},\quad e^{i\theta }\in\T\setminus\{1\}
\end{align*}
show that whenever $e^{i\theta }\in\T\setminus\{1\}$, $B_{n}(e^{i\theta })$ tends to $1$ as ${n\to+\infty }$. It follows that for every $e^{i\theta }\in\Omega \setminus \{1\}$, $g_{n}(e^{i\theta })$ tends to $f(e^{i\theta })$ as ${n\to+\infty }$. Since $|g_{n}(e^{i\theta })|=|f(e^{i\theta })|$ almost everywhere on $\T$ and $g_{n}\in L_{-}^{2}(\T)$ for each $n\ge 1$, by Lebesgue dominated convergence theorem we obtain that $f\in L^{2}_{-}(\T)$. But $f$ belongs to $H^{2}(\D)$ as well, and so $f=0$. We have thus proved that $\bigcup _{n\ge 0}K_{n}$ is dense in $H^{2}(\D)$. 
\par\medskip 
The next step of the proof is to show that $$\bigl( \bigcup_{n\ge 0}G_{n}\bigr)\cap\bigl( \bigcup_{n\ge 0}K_{n}\bigr)=\{0\}.$$ This is straightforward: if $f$ belongs to one of the spaces $K_{n}=B_{n}H^{2}(\D)$, $f$ vanishes at all the points $z_{j}$ for $j\ge n$, and $f$ has infinitely many zeroes. Hence $f$ cannot be a polynomial unless it vanishes identically.
\par\medskip 
Consider now the space $H=\bigoplus_{\ell_{2}}H^{2}(\D)$ which is the $\ell_{2}$-sum of countably many copies of $H^{2}(\D)$. This space can also be seen as the vector-valued $H^{2}$-space $H^{2}(\D,\ell_{2})$. Set, for each $n\ge 0$, $E_{n}=\bigoplus_{\ell_{2}}G_{n}$ and $F_{n}=\bigoplus_{\ell_{2}}K_{n}$. Obviously $E_{n}$ and $F_{n}$ are infinite-dimensional, $E_{n}\subseteq E_{n+1}$ and $F_{n}\subseteq F_{n+1}$ for each $n\ge 0$. Since $\dim (G_{n+1}\ominus G_{n})=\dim(K_{n+1}\ominus K_{n})=1$, it is clear that $E_{n}$ and $F_{n}$ are of infinite codimension in $E_{n+1}$ and $F_{n+1}$ respectively. Lastly, since $\bigcup _{n\ge 0}G_{n}$ and $\bigcup _{n\ge 0}K_{n}$ are both dense in $H^{2}(\D)$, $\bigcup _{n\ge 0}E_{n}$ and $\bigcup _{n\ge 0}F_{n}$ are dense in $H$. So (1), (2), and (3) of Proposition \ref{Prop7} are satisfied. It remains to see that $E_{n}\cap F_{m}=\{0\}$ for each $m, n\ge 0$, but this follows directly from the fact that $G_{n}\cap K_{m}=\{0\}$. This finishes the proof of Lemma \ref{Lem8}.
\end{proof}

Theorem \ref{Th1} is proved.
\par\smallskip
It is interesting to note  the following consequence of 
Theorem \ref{Th1} and the Lomonosov inequality:

\begin{proposition}
The two operators $T$ and $V^{-1}TV$ given by Theorem \ref{Th1}  are unitarily equivalent, and  the weakly (strongly) closed algebra $\mathcal{A}$ they generate is equal to $\mathcal{B}(H)$. The uniformly closed algebra $\mathcal{R}$ they generate contains all compact \ops\ on $H$. 
\end{proposition}

\begin{proof}
 First recall that the closures of a linear subspace of $\mathcal{B}(H)$ for the weak and strong topologies always coincide, so that we can indifferently consider $\mathcal{A}$ as the weak or strong closure of the algebra generated by $T$ and $V^{-1}TV$. Suppose that $\mathcal{A}\not = \mathcal{B}(H)$. Then we can apply the Lomonosov inequality, and obtain two non-zero vectors $x$ and $y$ of $H$ such that $|\langle x
, Ay\rangle|\le||A||_{e}$ for every $A\in \mathcal{A}$. So in particular we get that for every $n\ge 0$,
$|\langle x
, T^{n}y\rangle|\le||T^{n}||_{e}=||S_{0}^{n}||_{e}\le 1$, and $|\langle x
,  V^{-1}T^{n}Vy\rangle|\le||V^{-1}T^{n}V||_{e}=||V^{-1}S_{0}^{n}V||_{e}\le 1$. This contradicts the fact, given by Theorem \ref{Th1}, that either $y$ or $Vy$ has a dense orbit under the action of $T$. The statement about $\mathcal{R}$ is proved in exactly the same way, using \cite[Th. $3$]{L} which states that either there exist two non-zero vectors $x$ and $y$ of $H$ such that $|\langle x
, Ry\rangle|\le||R||_{e}$ for every $R\in \mathcal{R}$, or $\mathcal{R}$ contains all compact \ops\ on $H$.
\end{proof}

 There are of course much simpler examples of pairs $(R,S)$ of operators on $H$ which generate $\mathcal{B}(H)$: as was pointed out to us by Lyudmila Turowska, it suffices to take for $S$ the forward shift with respect to an orthonormal basis $(g_{j})_{j\ge 0}$ of $H$, and for $R$ the backward shift with respect to this same basis. Then $RS-SR$ is the rank one operator $\langle g_{0}\,,\, .\,\rangle g_{0}$, and from this one easily sees that the algebra generated by $R$ and $S$ is weakly (strongly) dense in $\mathcal{B}(H)$. This is a particular case of a result of \cite{MZ}, where it is shown that for any separable real or complex Banach space $X$, $\mathcal{B}(X)$ can be generated by two operators $R$ and $S$. Finding a minimal number of operators belonging to a certain subclass of $\mathcal{B}(H)$ and generating $\mathcal{B}(H)$ is a problem which has been much studied. We refer the reader for instance to the paper \cite{D}, where it is proved that $\mathcal{B}(H)$ can be generated by two unitary operators, and by no less than three projections, and to \cite{MZ} and the references therein.

\par\bigskip

\textbf{Acknowledgement:} We are grateful to Lyudmila Turowska for the remark above, and to Catalin Badea for pointing out references \cite{CTV}, \cite{D} and \cite{MZ}.

\end{document}